\documentclass[12pt,a4paper]{amsart}
\usepackage{amsmath, amsthm, amsfonts, amssymb, txfonts, graphicx, hyperref, bm,verbatim,siunitx}
\theoremstyle{plain}
\newtheorem{thrm}{Theorem}
\newtheorem{lmm}[thrm]{Lemma}
\newtheorem{prpstn}[thrm]{Proposition}

\newtheorem*{rmk}{Remark}

\numberwithin{sblmm}{thrm} 
\numberwithin{equation}{section}

\renewcommand{\phi}{\varphi}
\begin{document}
\title{Large gaps between primes}
\author{James Maynard}
\address{Magdalen College, Oxford OX1 4AU, England}
\email{james.alexander.maynard@gmail.com}
\begin{abstract}
We show that there exist pairs of consecutive primes less than $x$ whose difference is larger than 
\[t(1+o(1))(\log{x})(\log\log{x})(\log\log\log\log{x})(\log\log\log{x})^{-2}\]
for any fixed $t$. This answers a well-known question of Erd\H os.
\end{abstract}
\maketitle
\section{Introduction}
Let $G(X)=\sup_{p_n\le X}(p_{n+1}-p_n)$ denote the maximal gap between primes of size at most $X$. Westzynthius \cite{Westzynthius} was the first to show that $G(X)$ could become arbitrarily large compared with the average gap $(1+o(1))\log{X}$, and this was improved by Erd\H os \cite{Erdos} and then Rankin \cite{RankinOld}, who succeeded in showing that for $X$ sufficiently large
\begin{equation}
G(X)\ge (c+o(1))(\log{X})(\log_2{X})(\log_4{X})(\log_3{X})^{-2},\label{eq:RankinBound}
\end{equation}
with the constant $c=1/3$, where $\log_v$ denotes the $v$-fold logarithm. Since Rankin's 1938 result, however, the only improvements have been in the constant $c$. Such improvements were obtained in the work of Sch\"onhage \cite{Schonhage}, Rankin \cite{RankinNew}, Maier and Pomerance \cite{MaierPomerance}, with the best constant $c=2e^\gamma$ due to Pintz \cite{Pintz}. In this paper we show that one can take the constant $c$ to be arbitrarily large by incorporating sieve ideas based on the recent results on small gaps between primes \cite{Maynard,MaynardII,Polymath} into the Erd\H os-Rankin method.
\begin{thrm}\label{thrm:MainTheorem}
We have
\[\limsup_n\frac{p_{n+1}-p_n}{(\log{p_n})(\log_2{p_n})(\log_4{p_n})(\log_3{p_n})^{-2}}=\infty.\]
\end{thrm}
We note that Ford, Green, Konyagin and Tao have independently obtained this result in the recent work \cite{FGKT} using a different method. Their work is based on incorporating results on linear equations in the primes \cite{GTLinearEquations, GTInverse, GTMobius} (rather than work on small gaps between primes) into the Erd\H os-Rankin construction.
\begin{rmk}
The method presented here in fact allows one to obtain a quantitative improvement to Rankin's bound \eqref{eq:RankinBound}. We shall address this in forthcoming work.
\end{rmk}
\section{The Erd\H os-Rankin construction}\label{sec:ErdosRankin}
As with most approaches to the problem, we follow the Erd\H os-Rankin construction for large gaps, modifying only the final stage of the argument. We wish to choose residue classes $a_p\pmod{p}$ for each prime $p\le x$ such that every integer $n\in[1,U]$ satisfies $n\equiv a_p\pmod{p}$ for some prime $p\le x$.

We fix constants $C_U,\epsilon>0$, (and we will assume $\epsilon$ is sufficiently small at various parts of the argument) and let $y$, $z$, $U$ be defined in terms of $x$ by
\begin{align}
y=\exp\Bigl((1-\epsilon)\frac{\log{x}\log_3{x}}{\log_2{x}}\Bigr),\qquad z=\frac{x}{\log_2{x}},\qquad U=C_U\frac{x\log{y}}{\log_2{x}}.
\end{align}
The only difference between these choices and those of \cite{MaierPomerance} is that here $U$ is determined in terms of an unspecified constant $C_U$ (which we will show can be taken arbitrarily large) rather than a specific choice of $C_U$ slightly less than $1.32e^\gamma$.

If we can cover the interval $[1,U]$ by residue classes $a_p\pmod{p}$ for $p\le x$, then any $U_0$ satisfying $U_0\equiv-a_p\pmod{p}$ for all $p\le x$ has the property that $U_0+j$ has a prime factor of size at most $x$ for each $j\in[1,U]$. By the Chinese remainder theorem, there is such a $U_0$ in the interval $[x,x+\exp((1+o(1))x)]$, and this gives rise to an interval of length $U$ which contains no primes since $U_0+j$ is bigger than $x$ but contains a prime factor less than $x$ for all $j\in[1,U]$. Letting $x=(1-\epsilon)\log{X}$, and recalling our choice of $U$ and $y$ above, we see that this would show there is an interval in $[1,X]$ of length $(1-2\epsilon+o(1))C_U(\log{X})(\log_2{X})(\log_4{X})(\log_3{X})^{-2}$ containing no primes. Therefore we immediately obtain Theorem \ref{thrm:MainTheorem} if we can take $C_U$ to be arbitrarily large whilst still covering $[1,U]$ by the residue classes $a_p\pmod{p}$ for $p\le x$.

We choose $a_p$ for primes $p\le z$ by
\begin{equation}
\begin{split}
a_p&=1,\quad \text{for every prime $1< p\le y$,}\\
a_p&=0,\quad \text{for every prime $y<p\le z$.}
\end{split}
\end{equation}
After removing elements of $[1,U]$ in these residue classes we are left with the set
\begin{equation}
\mathcal{R}\cup\mathcal{R}',
\end{equation}
where
\begin{equation}
\begin{split}
\mathcal{R}&=\{mp\le U:p> z,m\text{ is $y$-smooth}, (mp-1,P_y)=1\},\\
\mathcal{R}'&=\{m\le U:m\text{ is $y$-smooth}, (m-1,P_y)=1\},\\
P_t&=\prod_{p\le t}p\qquad \text{for any $t\in\mathbb{R}$.}
\end{split}
\end{equation}
We first note that since $p>z$ the condition $(mp-1,P_y)=1$ requires that $m$ be even. We split $\mathcal{R}$ according to this integer $m$. For even $m$ we let
\begin{equation}
\mathcal{R}_m=\{z<p\le U/m:(mp-1,P_y)=1\}.
\end{equation}
\begin{lmm}\label{lmm:Rp}
We have
\[|\mathcal{R}'|\ll \frac{x}{(\log{x})^{1+\epsilon}}.\]
\end{lmm}
\begin{proof}
This is \cite[Theorem 5.3]{MaierPomerance}. Our slightly different choice of $U$ does not affect the argument from \cite{MaierPomerance}.
\end{proof}
\begin{lmm}\label{lmm:SetSizes}
We have uniformly for $z+z/\log{x}\le V\le x(\log{x})^2$ and $m\le x$
\begin{multline*}
\#\{z<p\le V:(mp-1,P_y)=1\}\\
=\frac{V-z}{\log{x}}\Bigl(\prod_{\substack{p\le y\\ p\nmid m}}\frac{p-2}{p-1}\Bigr)\Bigl(1+O(\exp(-(\log_2{x})^{1/2}))\Bigr).
\end{multline*}
In particular, uniformly for even $m\le U(1-1/\log{x})/z$ we have
\begin{align*} 
|\mathcal{R}_m|&=\frac{2 e^{-\gamma} U(1+o(1))}{m(\log{x})(\log{y})}\Bigl(\prod_{p>2}\frac{p(p-2)}{(p-1)^2}\Bigr)\Bigl(\prod_{p|m, p>2}\frac{p-1}{p-2}\Bigr).
\end{align*}
\begin{proof}
The first statement follows from a `fundamental lemma' sieve and the Bombieri-Vinogradov theorem; see \cite[Theorem 6.12]{FriedlanderIwaniec} for example. The second statement follows immediately from the first using Mertens' theorem.
\end{proof}
\begin{lmm}\label{lmm:RmBounds}
For any $M\ge 2$ we have
\[\sum_{U/(z M)\le m<U/z}|\mathcal{R}_m|\ll \frac{U \log{M}}{(\log{x})(\log{y})}.\]
In particular
\begin{align*}\sum_{U/(z(\log_2{x})^2)\le m<U/z}|\mathcal{R}_m|=o\Bigl(\frac{C_U x}{\log{x}}\Bigr),\qquad \sum_{1\le m<U/(z(\log_2{x})^2)}|\mathcal{R}_m|=O\Bigl(\frac{C_U x}{\log{x}}\Bigr).
\end{align*}
\end{lmm}
\begin{proof}
Let $w_1=U / (z M)$ and let $w_2=U(1-1/\log{x})/z$. For $m\ge w_2$ we use the trivial bound $|\mathcal{R}_m|\ll U/(m\log{x})$, and so we see that the contribution from $w_2\le m<U/z$ is at most $O(U/(\log{x})^2)$. We now consider $w_1\le m<w_2$. Using the bound
\[\prod_{p|m,\,p>2}\frac{p-1}{p-2}\ll \prod_{p|m}\frac{p+1}{p}=\sum_{d|m}\frac{1}{d},\]
letting $m=m' d$, and using Lemma \ref{lmm:SetSizes}, we have
\[\sum_{w_1\le m<w_2}|\mathcal{R}_m|\ll \frac{U}{(\log{x})(\log{y})}\sum_{d<w_2}\frac{1}{d^2}\sum_{w_1/d\le m'<w_2/d}\frac{1}{m'}\ll \frac{U(\log{M}+O(1))}{(\log{x})(\log{y})}.\]
The final estimates follow immediately from recalling the definitions of $U$ and $y$.
\end{proof}
\end{lmm}

We now state our key proposition.
\begin{prpstn}\label{prpstn:MainProp}
Fix $\delta>0$. Let $m<U z^{-1}(\log_2{x})^{-2}$ be even and let $\mathcal{I}_m\subseteq [x/2,x]$ be an interval of length at least $\delta |\mathcal{R}_m|\log{x}$.

Then for $x>x_0(\delta,C_U)$, there exists a choice of residue classes $a_q\pmod{q}$ for each prime $q\in\mathcal{I}_m$ such that
\[p\in\mathcal{R}_m\Rightarrow p\equiv a_q\pmod{q}\text{ for some prime $q\in\mathcal{I}_m$}.\]
\end{prpstn}
Theorem \ref{thrm:MainTheorem} now follows almost immediately from Proposition \ref{prpstn:MainProp}. 
\begin{proof}[Proof of Theorem \ref{thrm:MainTheorem} assuming Propostion \ref{prpstn:MainProp}]By Lemma \ref{lmm:RmBounds}, we see that
\begin{equation}
\sum_{m< U z^{-1}(\log_2{x})^{-2}}\delta |\mathcal{R}_m|\log{x}\ll \delta C_U x.
\end{equation}
Therefore, if $\delta$ is sufficiently small compared with $C_U$, we can choose intervals $\mathcal{I}_m$ of length $\delta|\mathcal{R}_m|\log{x}$ for each even $m< Uz^{-1}(\log_2{x})^{-2}$ such that all the $\mathcal{I}_m$ are disjoint and contained in $[x/2,x]$. By Proposition \ref{prpstn:MainProp} we can cover $\mathcal{R}_m$ using a residue class for each prime in $\mathcal{I}_m$, for each such $m$. By Lemmas \ref{lmm:Rp}, \ref{lmm:SetSizes} and \ref{lmm:RmBounds}, this means we can cover all but $o(x/\log{x})$ elements of $\mathcal{R}\cup\mathcal{R}'$ using only residue classes of primes in $[x/2,x]$. By choosing the residue class of one remaining element for each prime in $[z,x/2]$, we can then cover all the remaining elements of $\mathcal{R}\cup\mathcal{R}'$. Therefore we can choose residue classes $a_p\pmod{p}$ for all $p\le x$ which cover all of $[1,U]$, for any fixed choice of $C_U$. This completes the proof.
\end{proof}
We actually prove Proposition \ref{prpstn:MainProp} in a slightly different (but equivalent) form: We show that for any fixed $\epsilon,\delta>0$ and interval $\mathcal{I}_m\subseteq[x/2,x]$ of length at least $\delta|\mathcal{R}_m|\log{x}$, we can choose residue classes $a_q\pmod{q}$ for primes $q\in\mathcal{I}_m$ such that all but $\epsilon|\mathcal{R}_m|$ elements $p$ of $\mathcal{R}_m$ satisfy $p\equiv a_q\pmod{q}$ for some prime $q\in\mathcal{I}_m$. (Where $m$ is as in Proposition \ref{prpstn:MainProp}.)

By appending to $\mathcal{I}_m$ an interval of length $2\epsilon|\mathcal{R}_m|\log{x}$, we can choose the residue class of one of the remaining $\epsilon|\mathcal{R}_m|$ elements of $\mathcal{R}_m$ for each of the primes in this appended interval. Thus we can cover all of $\mathcal{R}_m$ by residue classes for primes in an interval of length $(2\epsilon+\delta)|\mathcal{R}_m|\log{x}$. Since $\epsilon$ and $\delta$ were arbitrary, we see these two forms of Proposition \ref{prpstn:MainProp} are equivalent.
\section{The probabilistic method}\label{sec:Probabilistic}
Given an even $m<U/(z(\log_2{x})^2)$ and a prime $q\in\mathcal{I}_m$, we will define a probability measure $\mu_{m,q}$ on the residue classes $a\pmod{q}$. We then consider the following situation: independently for each prime $q\in\mathcal{I}_m$, we randomly choose a residue class $a\pmod{q}$ with probability $\mu_{m,q}(a)$.

Given a prime $p\in\mathcal{R}_m$, we see that the probability that $p$ is not in any of the chosen residue classes for any $q\in\mathcal{I}_m$ is
\begin{align}
\prod_{q\in\mathcal{I}_m\text{ prime}}(1-\mu_{m,q}(p))&=\exp\Bigl(\sum_{q\in\mathcal{I}_m\text{ prime}}\log{(1-\mu_{m,q}(p))}\Bigr)\nonumber \\
&\le \exp\Bigl(-\sum_{q\in\mathcal{I}_m\text{ prime}}\mu_{m,q}(p)\Bigr).\label{eq:ProbBound}
\end{align}
Therefore if for almost every $p\in\mathcal{R}_m$ we have that the expected number of $q\in\mathcal{I}_m$ for which the residue class $p\pmod{q}$ is chosen is at least $t$, then the probability that any such $p\in\mathcal{R}_m$ is not in any of the chosen residue classes is less than $e^{-t}$. Therefore the expected number of primes in $\mathcal{R}_m$ which are not in any of the chosen residue classes is at most $e^{-t}|\mathcal{R}_m|$. If $t$ can be taken sufficiently large, we expect that all but at most $\epsilon|\mathcal{R}_m|$ elements of $\mathcal{R}_m$ are in at least one of the chosen residue classes. This means that there must be at least one configuration of residue classes $a\pmod{q}$ for $q\in\mathcal{I}_m$ which covers all but at most $\epsilon|\mathcal{R}_m|$ elements of $\mathcal{R}_m$, as required.
\section{GPY Probabilities}
We have seen that to complete the argument we require a probability measure $\mu_{m,q}$ for each prime $q\in\mathcal{I}_m$, such that for almost every $p\in\mathcal{R}_m$ the expected number $\sum_{q\in\mathcal{I}_m}\mu_{m,q}(p)$ of times  the residue class $p\pmod{q}$ is chosen is at least $t$, where $t$ can be taken to be arbitrarily large.

We wish $\mu_{m,q}(a)$ to be large when the residue class $a\pmod{q}$ contains many primes in $\mathcal{R}_m$, and small otherwise. The key feature in this situation is that the modulus $q$ is only slightly smaller than the size of the elements of $\mathcal{R}_m$, which makes it difficult to count the number of primes in a given residue class. To achieve such a measure, we adapt the weights used in \cite{Maynard, MaynardII} to this situation, so that $\mu_{m,q}(a)$ is large when $a\pmod{q}$ contains many elements with no small prime factors.

Specifically, first we choose an admissible set $\mathcal{H}=\{h_1,\dots,h_k\}$, with $h_j=p_{\pi(k)+j}P_{w}$ for each $j=1,\dots,k$ (i.e. $h_j$ is the $j^{th}$ prime greater than $k$, multiplied by all primes less than $w$). Here $w$ is a quantity which will go to infinity slowly with $x$, such that $P_w=o(\log_2{x})$ (we could take $w=\log_4{x}$, for example), and $k$ is a constant we will choose to be sufficiently large in terms of $\epsilon,\delta,C_U$. In particular $w$ will be large compared with $k$. We define
\begin{equation}
\mu_{m,q}(a)=\alpha_{m,q}\sum_{\substack{n\le U/m\\n\equiv a\pmod{q}\\ (n(m n-1),P_w)=1}}\Bigl(\sum_{\substack{d_1,\dots,d_k\\ d_i|n+h_i q}}\sum_{\substack{e_1,\dots,e_k\\ e_i|m(n+h_i q)-1}}\lambda_{d_1,\dots,d_k,e_1,\dots,e_k}\Bigr)^2.
\end{equation}
Here $\lambda_{d_1,\dots,d_k,e_1,\dots,e_k}$ are real constants (that we will choose later), and $\alpha_{m,q}$ is a normalizing constant so that $\sum_{a\pmod{q}}\mu_{m,q}(a)=1$.

The coefficients $\lambda_{d_1,\dots,d_k,e_1\dots,e_k}$ will factorize as $\lambda^{(1)}_{d_1,\dots,d_k}\prod_{i=1}^k\lambda_{e_i}^{(2)}$. The $\lambda^{(1)}_{d_1,\dots,d_k}$ correspond to a `GPY'  sieve, and ensure that $\mu_{m,q}(a)$ can only be large if there exists an $n\equiv a\pmod{q}$ such that all of $\{n+h_1q,\dots,n+h_k q\}$ have no small prime factors (and so we expect many of them to be prime). The $\lambda^{(2)}_{e_j}$ correspond to a standard Selberg sieve\footnote{We could apply a fundamental lemma type sieve here since $y=x^{o(1)}$, but we find it more convenient to apply a Selberg sieve, which is weaker only by an unimportant constant factor.}, and ensure that the contribution from such an $n$ is small unless $m(n+h_j q)-1$ has no prime factors less than $y^\epsilon$.

If we choose a residue class $a\pmod{q}$ randomly with probability $\mu_{m,q}(a)$, then for a suitable choice of $\lambda$ coefficients, we would find from following the work \cite{Maynard, MaynardII,Polymath} that the expected number of primes in $\mathcal{R}_m$ in the chosen residue class would be a constant multiple of $\log{k}$. One might hope that the primes found this way would be approximately independent for different $q\in\mathcal{I}_m$. If this were the case, then we would guess that the expected number of times a given prime in $\mathcal{R}_m$ would be in a picked residue class would be roughly the same for all primes in $\mathcal{R}_m$, in which case this would be approximately $(|\mathcal{I}_m|\log{k})/(|\mathcal{R}_m|\log{x})$, since there are roughly $|\mathcal{I}_m|/\log{x}$ primes in $\mathcal{I}_m$. Recalling that we choose $|\mathcal{I}_m|=\delta|\mathcal{R}_m|\log{x}$, we might therefore guess that the expected number of times $p\in\mathcal{R}_m$ is chosen is roughly a constant multiple of $\delta \log{k}$. (Normally this would actually depend on the arithmetic structure of $\mathcal{H},m,p_0$, but by choosing all elements of $\mathcal{H}$ to be a multiple of all small primes this effect is negligible.) Therefore, if $k$ is chosen sufficiently large, we expect to be able to make this quantity larger than any fixed constant. We now proceed to make these heuristic ideas rigorous.

In order for it to be feasible to estimate $\sum_q\mu_{m,q}(p)$, we exploit the linearity (in $n$ and $q$) of the expressions $n+h_i q$ and $m(n+h_i q)-1$, and make a choice of $\lambda_{d_1,\dots,d_k,e_1,\dots,e_k}$ which is independent of $q$. This allows us to estimate the resulting sums for fixed $n$ and varying prime $q$ and also for fixed $q$ and varying $n$. In particular, this makes it more convenient to adopt the `analytic' method for estimating the sums which appear, as in \cite{Polymath}.
\section{Setup}
We let
\begin{equation}
\begin{split}
\omega_{m,q}(p)=&\#\{1\le n\le p:n+h_i q\equiv 0\pmod{p}\\
&\qquad\text{ or }m(n+h_i q)\equiv 1\pmod{p}\text{ for some }1\le i\le k\},\label{eq:OmegaDef}
\end{split}
\end{equation}
and we extend $\omega_{m,q}$ to a totally multiplicative function defined on $\mathbb{N}$. Similarly, we define the totally multiplicative function $\phi_{m,q}$ by $\phi_{m,q}(p)=p-\omega_{m,q}(p)$. We put
\begin{equation}
\mathfrak{S}_{m,q}=\prod_{p\le y}\Bigl(1-\frac{\omega_{m,q}(p)}{p}\Bigr)\Bigl(1-\frac{1}{p}\Bigr)^{-2k},\label{eq:Smq}
\end{equation}
noting that this product is non-zero since $\omega_{m,q}(p)\le 2$ for $p\le w$ and $\omega_{m,q}(2)=1$ since we are only considering $m$ even.

We define $\lambda_{d_1,\dots,d_k,e_1\dots,e_k}$ by
\begin{equation}
\lambda_{d_1,\dots,d_k,e_1,\dots,e_k}=\Bigl(\prod_{i=1}^k\mu(d_i)\mu(e_i)\Bigr)\sum_{j=1}^J \Bigl(\prod_{\ell=1}^k F_{\ell,j}\Bigl(\frac{\log{d_\ell}}{\log{x}}\Bigr)G\Bigl(\frac{\log{e_\ell}}{\log{y}}\Bigr)\Bigr),\label{eq:LambdaDef}
\end{equation}
for some smooth non-negative functions $F_{i,j}$, $G:[0,\infty)\rightarrow\mathbb{R}$ which are not identically zero (which we declare later). The functions $F_{i,j},G$ and the quantity $J$ will be allowed to depend on $k$, but will be independent of $x,q$. Thus in particular $|\lambda_{d_1,\dots,d_k,e_1,\dots,e_k}|\ll_k 1$. We further require that for each $j\in\{1,\dots,J\}$, we have \begin{equation}
\sup\Bigl\{\sum_{i=1}^k u_i:F_{i,j}(u_i)\ne 0\Bigr\}\le 1/10,
\end{equation}
and we restrict $G$ to be supported on $[0,1]$. Finally, we put
\begin{equation}
F(t_1,\dots,t_k)=\sum_{j=1}^J\prod_{\ell =1}^k F'_{\ell,j}(t_\ell),
\end{equation}
and we assume that the $F_{\ell,j}$ are chosen such that $F$ is symmetric. We emphasize that this choice of $\lambda$ does not depend on $q$.

\section{Sieve estimates}
We first asymptotically evaluate the normalizing constant $\alpha_{m,q}$, and then estimate $\sum_{q\in\mathcal{I}_m}\mu_{m,q}(p)$.
\begin{lmm}\label{lmm:ConstSum}
We have
\[\alpha_{m,q}^{-1}= (1+o_k(1))\frac{U\mathfrak{S}_{m,q}}{m(\log{x})^{k}(\log{y})^{k}}I_k^{(1)}(F)I_k^{(2)}(G),\]
where $\mathfrak{S}_{m,q}$ is given by \eqref{eq:Smq}, and
\[I^{(1)}_k(F)=\idotsint\limits_{t_1,\dots,t_k\ge 0} F(t_1,\dots,t_k)^2dt_1\dots dt_k,\qquad I_k^{(2)}(G)=\Bigl(\int_0^\infty G'(t)^2dt\Bigr)^{k}.\]
\end{lmm}
\begin{proof}
The quantity $\alpha_{m,q}^{-1}$ is somewhat analogous to that of \cite[Proposition 9.1]{MaynardII}, although here we do not concern ourselves with uniformity in $k$. From the fact that we have defined $\alpha_{m,q}$ to be such that $\sum_{a\pmod{q}}\mu_{m,q}(a)=1$, we have that
\begin{align}
\alpha_{m,q}^{-1}=\sum_{\substack{n\le U/m\\ (n(mn-1),P_w)=1}}\Bigl(\sum_{\substack{d_1,\dots,d_k\\ d_i|n+h_iq}}\sum_{\substack{e_1,\dots,e_k\\ e_i|m(n+h_iq)-1}}\lambda_{d_1,\dots,d_k,e_1,\dots,e_k}\Bigr)^2.
\end{align}
Expanding the squares and swapping the order of summation shows this sum is equal to
\begin{align}
\sum_{\substack{d_1,\dots,d_k\\ d_1',\dots,d_k'}}\sum_{\substack{e_1,\dots,e_k\\ e_1',\dots,e_k'}}\lambda_{\mathbf{d},\mathbf{e}}\lambda_{\mathbf{d}',\mathbf{e}'}\sum_{\substack{n\le U/m\\ (n(m n-1),P_w)=1\\ [d_i,d_i']|n+h_i q\,\forall i\\ [e_i,e_i']|m(n+h_i q)-1\,\forall i}}1.\label{eq:ConstExpanded}
\end{align}
Here to ease notation we have written $\lambda_{\mathbf{d},\mathbf{e}}$ for $\lambda_{d_1,\dots,d_k,e_1\dots,e_k}$, and similarly for $\lambda_{\mathbf{d}',\mathbf{e}'}$. 

We concentrate on the inner sum. There is no contribution unless all of $d_1d_1'$, $\dots$, $d_k d_k'$, $e_1e_1'$, $\dots$, $e_k e_k'$ are coprime to $P_{w}$. Moreover, there is no contribution unless all of $d_1d_1',\dots,d_k d_k'$ are pairwise coprime, and all of $e_1e_1',\dots,e_k e_k'$ are pairwise coprime (since any $p|(e_i e_i',e_j e_j')$, say, would have to divide $(h_i-h_j)q$, which is in contradiction to $p\le x^{1/10}<q$ being prime and $(e_i e_i',P_{w})=1$). Finally, we see that we must have $(d_i d_i',e_j e_j')|m q(h_j-h_i)-1$ for all $1\le i, j\le k$. If all of these conditions are satisfied, then, by the Chinese remainder theorem, we can combine the divisibility conditions in the inner sum to restrict $n$ to lie in any of $\phi_{m,q}(P_{w})$  residue classes $\pmod{P_{w}[d_1,d_1',e_1,e_1',\dots,d_k,d_k',e_k,e_k']}$.

Thus we see that \eqref{eq:ConstExpanded} is given by
\begin{align}
\sideset{}{'}\sum_{\substack{d_1,\dots,d_k\\ d_1',\dots,d_k'}}\sideset{}{'}\sum_{\substack{e_1,\dots,e_k\\ e_1',\dots,e_k'}}\lambda_{\mathbf{d},\mathbf{e}}\lambda_{\mathbf{d}',\mathbf{e}'}\Bigl(\frac{U\phi_{m,q}(P_{w})}{m[\mathbf{d},\mathbf{d}',\mathbf{e},\mathbf{e}']P_{w}}+O(\phi_{m,q}(P_{w}))\Bigr),
\end{align}
where we have written $[\mathbf{d},\mathbf{d}',\mathbf{e},\mathbf{e}']$ for $[d_1,d_1',e_1,e_1',\dots,d_k,d_k',e_k,e_k']$, and $\sum'$ for the conditions that $d_1 d_1',\dots,d_k d_k',P_{w}$ are pairwise coprime, $e_1e_1'$, $\dots$, $e_k e_k'$, $P_{w}$ are pairwise coprime, and that $(e_i e_i',d_j d_j')|(m q(h_j-h_i)-1,P_y)$ for all $1\le i, j\le k$. The fact we have forced this common factor to divide $P_y$ is redundant since $\lambda_{\mathbf{d},\mathbf{e}}$ is supported only on $e_i\le y$, but is slightly convenient later.

We first estimate the error term trivially. We have $\sup_{\mathbf{d},\mathbf{e}}(|\lambda_{\mathbf{d},\mathbf{e}}|)\ll_k 1$. We write $d=\prod_{i=1}^k d_i$ (and similarly for $d',e,e'$), and see that the support conditions of $F_{i,j},G$ mean we only need to consider $d,d'\le x^{1/10}$ and $e,e'\le y^k\ll_k x^\epsilon$. Therefore the error term contributes
\begin{equation}
\ll_k \phi_{m,q}(P_{w})\sum_{d,d'\le x^{1/10}, e,e'\ll_k x^\epsilon}\tau_k(d)\tau_k(d')\tau_k(e)\tau_k(e')\ll_k x^{1/2}.
\end{equation}
We now estimate the main term. The argument we use is standard, and is a minor adaption of \cite[Lemma 4.1]{Polymath}. 

We expand $\lambda_{\mathbf{d},\mathbf{e}}$, $\lambda_{\mathbf{d}',\mathbf{e}'}$ using \eqref{eq:LambdaDef}. Thus we are left to evaluate
\begin{align}
\sum_{j=1}^J\sum_{j'=1}^J \sideset{}{'}\sum_{\substack{d_1,\dots,d_k\\ d_1',\dots,d_k'}}\sideset{}{'}\sum_{\substack{e_1,\dots,e_k\\ e_1',\dots,e_k'}}\frac{\prod_{\ell =1}^k\mu(d_\ell)\mu(d_\ell')\mu(e_\ell)\mu(e_\ell')H_{\ell,j,j'}(d_\ell,d'_\ell,e_\ell,e'_\ell)}{[\mathbf{d},\mathbf{d}',\mathbf{e},\mathbf{e}']},\label{eq:ConstBasicMainTerm}
\end{align}
where
\[
H_{\ell,j,j'}(d_\ell,d'_\ell,e_\ell,e'_\ell)=F_{\ell ,j}\Bigl(\frac{\log{d_\ell}}{\log{x}}\Bigr)F_{\ell ,j'}\Bigl(\frac{\log{d_\ell'}}{\log{x}}\Bigr)G\Bigl(\frac{\log{e_\ell}}{\log{y}}\Bigr)G\Bigl(\frac{\log{e_\ell'}}{\log{y}}\Bigr).
\]
The function $e^{t}F_{\ell,j}(t)$ can be extended to a smooth compactly supported function on $\mathbb{R}$, and so has a Fourier expansion $e^{t}F_{\ell,j}(t)=\int_\mathbb{R}e^{-it\xi}f_{\ell ,j}(\xi)d\xi$, for a function $f_{\ell,j}$ which (from the smoothness of $e^t F_{\ell,j}(t)$ and integration by parts) satisfies $f_{\ell,j}(\xi)\ll_{k,A} (1+|\xi|)^{-A}$ for any $A>0$, and so is rapidly decreasing. In particular, we have
\begin{equation}
F_{\ell,j}\Bigl(\frac{\log{d_\ell}}{\log{x}}\Bigr)=\int_{\mathbb{R}}\frac{f_{\ell,j}(\xi_{\ell})}{d_\ell^{(1+i\xi_{\ell})/\log{x}}}d\xi_{\ell}.
\end{equation}
We obtain an analogous expression for $G$. Thus the sum over the $d$, $d'$, $e$, and $e'$ variables in \eqref{eq:ConstBasicMainTerm} can then be rewritten as
\begin{align}
&\int_{\mathbb{R}}\dotsi\int_{\mathbb{R}}\Biggl(\sideset{}{'}\sum_{\substack{d_1,\dots,d_k\\ d_1',\dots,d_k'}}\sideset{}{'}\sum_{\substack{e_1,\dots,e_k\\ e_1',\dots,e_k'}}\frac{1}{[\mathbf{d},\mathbf{d}',\mathbf{e},\mathbf{e}']}\prod_{\ell=1}^k\frac{\mu(d_\ell)\mu(d_\ell')\mu(e_\ell)\mu(e_\ell')}{d_\ell^{\frac{1+i\xi_\ell}{\log{x}}}(d_\ell')^{\frac{1+i\xi_\ell'}{\log{x}}}e_\ell^{\frac{1+i\tau_\ell}{\log{y}}}(e_\ell')^{\frac{1+i\tau_\ell'}{\log{y}}}}\Biggr)\nonumber\\
&\qquad\times\Bigl(\prod_{\ell =1}^k f_{\ell,j}(\xi_\ell)f_{\ell,j'}(\xi_{\ell}')g(\tau_\ell)g(\tau_\ell')d\xi_\ell d\xi_\ell'd\tau_\ell d\tau_\ell'\Bigr).
\end{align}
Here we have swapped the order of summation and integration (noting the expression is absolutely convergent).

We concentrate on the first term in parentheses in the integral. Since the restrictions imposed on the summation are multiplicative and the summand is also multiplicative, we can rewrite the sum as a product $\prod_{p}K_p$ for functions $K_p(\xi_1\dots,\xi_k,\xi_1',\dots,\xi_k',\tau_1,\dots,\tau_k,\tau_1',\dots,\tau_k')$. We first notice that
\begin{align}
K_p&=\sideset{}{'}\sum_{\substack{d_1,\dots,d_k\\ d_1',\dots,d_k'}}\sideset{}{'}\sum_{\substack{e_1,\dots,e_k\\ e_1',\dots,e_k'\\ [\mathbf{d},\mathbf{d}',\mathbf{e},\mathbf{e}']|p}}\frac{1}{[\mathbf{d},\mathbf{d}',\mathbf{e},\mathbf{e}']}\prod_{\ell=1}^k\frac{\mu(d_\ell)\mu(d_\ell')\mu(e_\ell)\mu(e_\ell')}{d_\ell^{\frac{1+i\xi_\ell}{\log{x}}}(d_\ell')^{\frac{1+i\xi_\ell'}{\log{x}}}e_\ell^{\frac{1+i\tau_\ell}{\log{y}}}(e_\ell')^{\frac{1+i\tau_\ell'}{\log{y}}}}\nonumber\\
&=1+O_k(p^{-1-1/\log{x}}).\end{align}
Thus $\prod_{p}K_p\ll (\log{x})^{O_k(1)}$. Since all the $f,g$ functions are rapidly decreasing, this means that we can restrict the integral to $|\xi_\ell|$, $|\xi_\ell'|$, $|\tau_\ell|$, $|\tau_\ell'|\le \sqrt{\log{x}}$ for all $\ell$ at the cost of an error $O_k((\log{x})^{-2k})$. To ease notation we let $s_j=(1+i\xi_j)/\log{x}$, $r_\ell=(1+i\tau_\ell)/\log{y}$ and similarly for $s_j',r_\ell'$.

For $w<p\le y$ with $p\nmid \prod_{h,h'\in\mathcal{H}}(m q(h-h')-1)$, or for $p>y$, we have 
\begin{align}
K_p&=\Bigl(1+O_k\Bigl(\frac{1}{p^{2}}\Bigr)\Bigr)\prod_{\ell=1}^k\frac{\Bigl(1-p^{-1-s_\ell}\Bigr)\Bigl(1-p^{-1-s_\ell'}\Bigr)\Bigl(1-p^{-1-r_\ell}\Bigr)\Bigl(1-p^{-1-r_\ell'}\Bigr)}{\Bigl(1-p^{-1-s_\ell-s_\ell'}\Bigr)\Bigl(1-p^{-1-r_\ell-r_\ell'}\Bigr)}.\label{eq:GoodSmallPrimes}
\end{align}
For $w<p\le y$ with $p| \prod_{h,h'\in\mathcal{H}}(m q(h-h')-1)$, we have terms involving the product of $d_j$ and $e_\ell$ if $p|m q(h_\ell-h_j)-1$. This means that we have an additional factor compared with \eqref{eq:GoodSmallPrimes} of
\begin{align}
&\Bigl(1+O_k(p^{-2})\Bigr)\prod_{j,\ell:\, p|m q(h_\ell-h_j)-1}\Bigl(1+\sum_{\substack{\mathcal{T}\subseteq\{s_j,s_j',r_\ell,r_\ell'\}\\ \mathcal{T}\cap\{s_j,s_j'\}\ne \emptyset\\\mathcal{T}\cap\{r_j,r_j'\}\ne \emptyset}}(-1)^{\#\mathcal{T}}p^{-\sum_{t\in\mathcal{T}}t}\Bigr)\nonumber\\
&=\Bigl(1+\frac{\#\{j,\ell:\,p|m q(h_\ell-h_j)-1\}}{p}\Bigr)\Bigl(1+O_k\Bigl(\frac{1}{p^{2}}+\frac{\log{p}\sqrt{\log{x}}}{p\log{y}}\Bigr)\Bigr).\label{eq:BadPrimes}
\end{align}
Here we used the fact that we have restricted to $|s_j|$, $|r_\ell|<(\log{x})^{1/2}/\log{y}$ (so, for example, $p^{-s_j}=1+O((\log{p})(\log{x})^{1/2}/(\log{y}))$ by Taylor expansion) to obtain the final error term.

For such $p$, we see that the $h_i q\pmod{p}$ are all distinct ($p>w$ implies $p\nmid \prod_{j\ne\ell}(h_j-h_\ell)$ and $p\le y<q$ implies $p\nmid q$). Therefore, recalling the definition \eqref{eq:OmegaDef} of $\omega_{m,q}(p)$, we see that $\omega_{m,q}(p)=2k-\#\{j,\ell:\, p|m q(h_\ell-h_j)-1\}$, and so the first factor in \eqref{eq:BadPrimes} simplifies to $1-(\omega_{m,q}(p)-2k)/p$.

We see that since $\prod_{h,h'\in\mathcal{H}}(m q(h-h')-1)\ll x^{O_k(1)}$, and $\log{y}>(\log{x})^{1-\epsilon}$, we have
\begin{align}
\prod_{\substack{p|\prod_{h,h'}(m q(h-h')-1)\\p>w}}&\Bigl(1+O_k\Bigl(\frac{1}{p^2}+\frac{\log{p}\sqrt{\log{x}}}{p\log{y}}\Bigr)\Bigr)\nonumber\\
&=\exp\Bigl(O_k\Bigl(w^{-1}+\frac{\log\log{x}}{(\log{x})^{1/2-\epsilon}}\Bigr)\Bigr)\nonumber\\
&=1+o_k(1),
\end{align}
and so the second factor in \eqref{eq:BadPrimes} has a negligible effect. 

Finally, we see that since 
\[\log{w}=o((\log{x})^\epsilon)\text{ and }s_\ell,s_\ell',r_\ell,r_\ell'=o(\log{x})^{-1/2+\epsilon},\]
 we have
\begin{multline}
\prod_{p\le w}\prod_{\ell=1}^k\frac{\Bigl(1-p^{-1-s_\ell}\Bigr)\Bigl(1-p^{-1-s_\ell'}\Bigr)\Bigl(1-p^{-1-r_\ell}\Bigr)\Bigl(1-p^{-1-r'_\ell}\Bigr)}{\Bigl(1-p^{-1-s_\ell-s_\ell'}\Bigr)\Bigl(1-p^{-1-r_\ell-r_\ell'}\Bigr)}\\
=(1+o_k(1))\prod_{p<w}\Bigl(1-\frac{1}{p}\Bigr)^{2k}.
\end{multline}
Putting this all together gives
\begin{align}
\prod_{p>w}K_p&=(1+o_k(1))\prod_{p<w}\Bigl(1-\frac{1}{p}\Bigr)^{-2k}\prod_{w<p\le y}\Bigl(1-\frac{\omega_{m,q}(p)-2k}{p}\Bigr)\nonumber\\
&\qquad\times\prod_{\ell=1}^k\frac{\zeta\Bigl(1+\frac{2+i\xi_\ell+i\xi_\ell'}{\log{x}}\Bigr)\zeta\Bigl(1+\frac{2+i\tau_\ell+i\tau_\ell'}{\log{y}}\Bigr)}{\zeta\Bigl(1+\frac{1+i\xi_\ell}{\log{x}}\Bigr)\zeta\Bigl(1+\frac{1+i\xi_\ell'}{\log{x}}\Bigr)\zeta\Bigl(1+\frac{1+i\tau_\ell}{\log{y}}\Bigr)\zeta\Bigl(1+\frac{1+i\tau_\ell'}{\log{y}}\Bigr)}.
\end{align}
Here we have extended the product of $1-(\omega_{m,q}(p)-2k)/p$ to all primes $w<p\le y$, which is valid since $\omega_{m,q}(p)=2k$ if $p\nmid \prod_{h,h'}(mq(h-h')-1)$ for such $p$.

In the region $|z|=o(1)$, we have the estimate $\zeta(1+z)=(1+o(1))/z$. Thus, recalling that $\log{y}\ge (\log{x})^{1-\epsilon}$, we are left to estimate
\begin{align}
\idotsint(1+o_k(1))&\prod_{\ell =1}^k\Biggl(\frac{(1+i\xi_\ell)(1+i\xi_\ell')(1+i\tau_\ell)(1+i\tau_\ell')}{(2+i\xi_\ell+i\xi_\ell')(2+i\tau_\ell+i\tau_\ell')}\nonumber\\
&\times  f_{\ell,j}(\xi_\ell)f_{\ell,j'}(\xi_{\ell}'g(\tau_\ell)g(\tau_\ell'))d\xi_\ell d\xi_\ell'd\tau_\ell d\tau_\ell'\Biggr),
\end{align}
where the integral is over $|\xi_\ell|$, $|\xi'_\ell|$, $|\tau_\ell|$, $|\tau'_\ell|\le \sqrt{\log{x}}$. From the rapid decay of the $f$ and $g$ functions, we see that the $o_k(1)$ term contributes $o_k(1)$ in total, and we can extend the integrals to being over $\mathbb{R}$ at a cost of $o_k(1)$. Thus it suffices (since the integrals are absolutely convergent) to show that for any $f_1,f_2$ amongst the $f_{\ell,j}$, $g$ we have
\begin{align}
\int_\mathbb{R}\int_\mathbb{R}\frac{(1+i\xi)(1+i\xi')}{2+i\xi+i\xi'}f_1(\xi)f_2(\xi')d\xi d\xi'=\int_0^\infty F_1'(t)F_2'(t)dt.
\end{align}
This follows immediately from our definition of the Fourier transform (differentiating under the integral sign is valid due to absolute convergence).

Putting everything together, we have that
\begin{align}
\alpha_{m,q}^{-1}&=\frac{(1+o_k(1))U\phi_{m,q}(P_w)}{m(\log{x})^k(\log{y})^k P_w}\prod_{p\le w}\Bigl(1-\frac{1}{p}\Bigr)^{-2k}\prod_{w\le p\le y}\Bigl(1-\frac{\omega_{m,q}(p)-2k}{p}\Bigr)\nonumber\\
&\qquad \times\Bigl(\int_0^\infty G'(t)^2dt\Bigr)^{k}\sum_{j=1}^J\sum_{j'=1}^J\prod_{\ell=1}^k\Bigl(\int_0^\infty F_{\ell,j}'(t)F_{\ell,j'}'(t)dt\Bigr)\nonumber\\
&=(1+o_k(1))\frac{U\mathfrak{S}_{m,q}I_k(F)}{m(\log{x})^k(\log{y})^k }I_k^{(1)}(F)I_k^{(2)}(G).
\end{align}
Here we have used the fact that $G$ and the $F_{\ell,j}$ are all non-negative (and not the zero function) to take the $o(1)$ errors as a factor at the front of the expression.
\end{proof}
\begin{lmm}\label{lmm:PrimeSum}
Let $m<U z^{-1}(\log{x})^{-2}$ be even and let $p_0\in\mathcal{R}_m$ with $h_k x<p_0<U/m-h_k x$. Then
\[\sum_{q\in\mathcal{I}_m}\mu_{m,q}(p_0)\gg (1+o_k(1))\frac{k|\mathcal{I}_m|J_k^{(1)}(F)J_k^{(2)}(G)}{(\log{x})|\mathcal{R}_m| I_k^{(1)}(F)I_k^{(2)}(G)},\]
where
\begin{align*}
J_k^{(1)}(F)&=\idotsint\limits_{t_1,\dots,t_{k-1}\ge 0}\Bigl(\int_{t_k\ge 0}F(t_1,\dots,t_k)dt_k\Bigr)^2dt_1\dots dt_{k-1},\\
J_k^{(2)}(G)&=G(0)^2\Bigl(\int_0^\infty G'(t)^2dt\Bigr)^{k-1}.
\end{align*}
\end{lmm}
\begin{proof}
We substitute the definition of $\mu_{m,q}$ to give
\begin{align}
\sum_{q\in\mathcal{I}_m\text{ prime}}&\mu_{m,q}(p_0)\nonumber\\
&=\sum_{q\in\mathcal{I}_m\text{ prime}}\alpha_{m,q}\sum_{\substack{n\le U/m\\n\equiv p_0\pmod{q}\\ 
(n(m n-1),P_w)=1}}\Bigl(\sum_{\substack{d_1,\dots,d_k\\ d_i|n+h_i q}}\sum_{\substack{e_1,\dots,e_k\\ e_i|m(n+h_i q)-1}}\lambda_{d_1,\dots,d_k,e_1,\dots,e_k}\Bigr)^2.
\end{align}
Since all terms are non-negative, we obtain a lower bound by dropping all terms in the sum over $n$ except for when $n=p_0-h q$ for some $h\in\mathcal{H}$. We see that $(p_0+(h_i-h)q,P_{w})=(p_0,P_w)=1$ and $(mp_0+m(h_i-h)q-1,P_w)=(mp_0-1,P_w)=1$ so all the terms $n=p_0-h q$ appear in the sum (since, by assumption, $h_k x<p_0<U/m-h_k x$). This gives
\begin{align}
\sum_{q\in\mathcal{I}_m\text{ prime}}&\mu_{m,q}(p_0)\nonumber\\
&\ge \sum_{h\in\mathcal{H}} \sum_{\substack{q\in\mathcal{I}_m\text{ prime}
}}
\alpha_{m,q}\Bigl(\sum_{\substack{d_1,\dots,d_k\\ d_i|p_0+(h_i-h)q}}\sum_{\substack{e_1,\dots,e_k\\ e_i|m(p_0+(h_i-h)q)-1}}\lambda_{d_1,\dots,d_k,e_1,\dots,e_k}\Bigr)^2.\label{eq:MainExpression}
\end{align}
We split the sum over $q$ into residue classes modulo $P_{w}$. This gives
\begin{equation}
\sum_{h\in\mathcal{H}} \sum_{\substack{w_0\text{ (mod $P_w$)}
\\ (w_0,P_w)=1}}\sum_{\substack{q\in\mathcal{I}_m\text{ prime}\\ q\equiv w_0\text{ (mod $P_w$)}}}\alpha_{m,q}\Bigl(\sum_{\substack{d_1,\dots,d_k\\ d_i|p_0+(h_i-h)q}}\sum_{\substack{e_1,\dots,e_k\\ e_i|m(p_0+(h_i-h)q)-1}}\lambda_{d_1,\dots,d_k,e_1,\dots,e_k}\Bigr)^2.
\end{equation}
We now replace $\alpha_{m,q}$ with a slightly more manageable expression with less dependence on $q$. We first note that since all the $h_i$ are a multiple of $P_w$, for every prime $p\le w$ we have that $\omega_{m,q}(p)=2$ or $1$ depending on whether or not $(m,p)=1$. For $w<p\le y$ we have $\omega_{m,q}(p)=2k$ if 
\[p\nmid \prod_{h',h''\in\mathcal{H}}(mq(h'-h'')-1).\]
Thus, recalling that we only consider even $m$, we have
\begin{align}
\mathfrak{S}_{m,q}^{-1}&=\prod_{p\le y}\Bigl(1-\frac{\omega_{m,q}(p)}{p}\Bigr)^{-1}\Bigl(1-\frac{1}{p}\Bigr)^{2k}\nonumber\\
&\ge (1+o_k(1))\mathfrak{S}_m\prod_{\substack{w<p\le y\\ p|\prod_{h',h''}(m q(h'-h'')-1)}}\Bigl(1-\frac{2k}{p}\Bigr),\label{eq:SmqInv}
\end{align}
where we have put
\begin{equation}
\mathfrak{S}_m=2^{-(2k-1)}\Bigl(\prod_{p|m,\, p>2}\frac{p-2}{p-1}\Bigr)\prod_{2<p\le w}\Bigl(1-\frac{2}{p}\Bigr)^{-1}\Bigl(1-\frac{1}{p}\Bigr)^{2k}.\label{eq:SmDef}
\end{equation}
Since $\prod_{h',h''}(m q(h'-h'')-1)\ll x^{O_k(1)}$, we can restrict the primes occurring in the final product on the right hand side of \eqref{eq:SmqInv} to be less than $z_0=\log{x}/\log_2{x}$ at a cost of a factor of $1+o_k(1)$. Expanding out this product then gives
\begin{align}
\mathfrak{S}_{m,q}^{-1}&\ge (1+o_k(1))\mathfrak{S}_m\sum_{\substack{a_{1,2},\dots,a_{k,k-1}|P_{z_0}/P_{w}\\ a_{i,j}|m q(h_i-h_j)-1}}\frac{(-2k)^{\omega([\mathbf{a}])}}{[\mathbf{a}]}.
\end{align}
Here we have put $[\mathbf{a}]=[a_{1,2},\dots,a_{k,k-1}]$. Substituting this bound for $\mathfrak{S}_{m,q}$ and our bound for $\alpha_{m,q}$ from Lemma \ref{lmm:ConstSum} into \eqref{eq:MainExpression}, we have
\begin{align}
\sum_{q\in\mathcal{I}_m\text{ prime}}\mu_{m,q}(p_0)&\ge \frac{(1+o_k(1))m(\log{x})^k(\log{y})^k\mathfrak{S}_m }{U I_k^{(1)}(F) I_k^{(2)}(G)}\nonumber\\
&\times\sum_{h\in\mathcal{H}}\sum_{\substack{w_0\text{ (mod $P_w$)}\\ 
(w_0,P_w)=1}}
\sum_{a_{1,2},\dots,a_{k,k-1}|P_{z_0}/P_{w}}\frac{(-2k)^{\omega([\mathbf{a}])}}{[\mathbf{a}]} 
\nonumber\\ 
&\times\sum_{\substack{q\in\mathcal{I}_m\text{ prime}\\ q\equiv w_0\text{ (mod $P_w$)}\\ a_{i,j}|mq(h_i-h_j)-1\,\forall i\ne j}}\Bigl(\sum_{\substack{d_1,\dots,d_k\\ d_i|p_0+(h_i-h)q}}\sum_{\substack{e_1,\dots,e_k\\ e_i|m(p_0+(h_i-h)q)-1}}\lambda_{d_1,\dots,d_k,e_1,\dots,e_k}\Bigr)^2.\label{eq:MainExpression2}
\end{align}
We concentrate on the sum over $q$. For convenience we will consider the case when $h$ in the outer sum is $h_k$; the other cases are entirely analogous.

Since $p_0$ is a prime larger than $x$, and $(mp_0-1,P_{y})=1$, we may restrict to $d_k=e_k=1$ since no other divisors of $p_0$ or $mp_0-1$ occur. Inserting this condition, expanding the square and swapping the order of summation then gives that the sum over $q$ is equal to
\begin{align}
\sum_{\substack{d_1,\dots,d_k\\ d_1',\dots, d_k'\\ d_k=d_k'=1}}\sum_{\substack{e_1,\dots,e_k\\ e_1',\dots, e_k'\\ e_k=e_k'=1}}\lambda_{\mathbf{d},\mathbf{e}}\lambda_{\mathbf{d}',\mathbf{e}'}\sum_{\substack{q\in\mathcal{I}_m\text{ prime}\\ [d_i,d_i']|p_0+(h_i-h_k)q\,\forall i\\ [e_i,e_i']|mp_0+m(h_i-h_k)q-1\,\forall i \\ q\equiv w_0\text{ (mod $P_w$)}\\ a_{i,j}|m q(h_i-h_j)-1\,\forall i\ne j}}1.\label{eq:qSum}
\end{align}
This is now an expression which is similar to that considered in the proof of \cite[Proposition 9.2]{MaynardII}. Let us be given $a_{1,2},\dots,a_{k,k-1}|P_{z_0}/P_w$ and $\mathbf{d},\mathbf{d}',\mathbf{e},\mathbf{e}'$ with $\lambda_{\mathbf{d},\mathbf{e}}\lambda_{\mathbf{d}',\mathbf{e}'}\ne 0$ and with $d_k=d_k'=e_k=e_k'=1$. We see that the inner sum over $q$ is empty unless $d_1d_1',\dots,d_k d_k',P_w$ are pairwise coprime, $e_1e_1',\dots,e_k e_k',P_w$ are pairwise coprime and $a_{1,2},\dots,a_{k,k-1},m$ are pairwise coprime. Moreover, we must also have that 
\begin{align}
(d_i d_i',e_j e_j')&\,|\,mp_0(h_i-h_j)+h_k-h_i,&&\forall i,j,&\nonumber\\
(d_i d_i',a_{j,\ell})&\,|\,(h_j-h_\ell)mp_0+h_i-h_k,&&\forall i,j,\ell,&\label{eq:Div1}\\
(e_i e_i',a_{j,\ell})&\,|\,(h_j-h_\ell)(1-p_0m)-h_i+h_k,&&\forall i,j,\ell.\nonumber&
\end{align}
If all of these conditions are satisfied, then the inner sum can be rewritten as a sum over primes in $\mathcal{I}_m$ in a single residue class modulo the least common multiple of $d_1$, $d_1'$, $e_1$, $e_1'$,$\dots$, $d_k$, $d_k'$, $e_k$, $e_k'$, $a_{1,2}$ $\dots$, $a_{k,k-1}$, $P_w$. Moreover, this residue class will be coprime to the modulus. Thus in this case the inner sum is
\begin{equation}
\frac{\sum_{q\in\mathcal{I}_m\text{ prime}}1}{\phi(P_{w})\phi([\mathbf{d},\mathbf{d}',\mathbf{e},\mathbf{e}',\mathbf{a}])}+O(E(x;P_w[\mathbf{d},\mathbf{d}',\mathbf{e},\mathbf{e}',\mathbf{a}])),
\end{equation}
where we have written $[\mathbf{d},\mathbf{d}',\mathbf{e},\mathbf{e}',\mathbf{a}]$ to represent the least common multiple of all the $d$, $d'$, $e$, $e'$ and $a$ variables, and where
\begin{equation}
E(x;q)=\sup_{t\le x}\sup_{(a,q)=1}\Bigl|\pi(t;q,a)-\frac{\pi(t)}{\phi(q)}\Bigr|.\end{equation}
We put $r=P_{w}[\mathbf{d},\mathbf{d}',\mathbf{e},\mathbf{e}',\mathbf{a}]$, and note that $r\ll x^{1/5+o_k(1)}$, from the support conditions on $F_{\ell,j}$ and $G$ and the fact that $a_{j,l}|P_{z_0}=x^{o(1)}$. Therefore, using the trivial bound $E(x;q)\ll x/q$ and the Bombieri-Vinogradov theorem, we see the contribution of the error to \eqref{eq:qSum} is
\begin{align}
&\ll_k \sum_{r\ll x^{1/5+\epsilon}}\tau_{k^2+4k}(r)E(x;r)\nonumber\\
&\ll \Bigl(x\sum_{r\ll x^{1/5+\epsilon}}\frac{\tau_{k^2+4k}(r)^2}{r}\Bigr)^{1/2}\Bigl(\sum_{r\ll x^{1/5+\epsilon}}E(x;r)\Bigr)^{1/2}\ll_k \frac{x}{(\log{x})^{2k}},
\end{align}
which will be negligible (here we used $|\lambda_{\mathbf{d},\mathbf{e}}|\ll_k 1$ for all $\mathbf{d},\mathbf{e}$). 

$\mathcal{I}_m$ is an interval of length $\delta|\mathcal{R}_m|\log{x}\gg x(\log{x})^{-2}$ by Lemma \ref{lmm:SetSizes} and our bound on $m$. Since $\mathcal{I}_m$ is contained in $[x/2,x]$, the number of primes in $\mathcal{I}_m$ is $(1+o(1))|\mathcal{I}_m|/\log{x}$ by the prime number theorem. Therefore the main term of \eqref{eq:qSum} then simplifies to
\begin{align}
&\frac{(1+o(1))|\mathcal{I}_m|}{\phi(P_{w})\log{x}}\sideset{}{^*}\sum_{\substack{d_1,\dots,d_k\\ d_1',\dots,d_k'}}\,\,\sideset{}{^*}\sum_{\substack{e_1,\dots,e_k\\ e_1',\dots,e_k'}}\frac{\lambda_{\mathbf{d},\mathbf{e}}\lambda_{\mathbf{d}',\mathbf{e}'}}{\phi([\mathbf{d},\mathbf{d}',\mathbf{e},\mathbf{e}',\mathbf{a}])}\nonumber\\
&=\frac{(1+o(1))|\mathcal{I}_m|G(0)^2}{\phi(P_{w})\log{x}}\sum_{j=1}^J\sum_{j'=1}^J F_{k,j}(0)F_{k,j'}(0)\sideset{}{^{**}}\sum_{\substack{d_1,\dots,d_{k-1}\\ d_1',\dots,d_{k-1}'}}\,\,\,\sideset{}{^{**}}\sum_{\substack{e_1,\dots,e_{k-1}\\ e_1',\dots,e_{k-1}'}}\nonumber\\
&\times\frac{\prod_{\ell =1}^{k-1}\mu(d_\ell)\mu(d_\ell')\mu(e_\ell)\mu(e_\ell')F_{\ell ,j}\Bigl(\frac{\log{d_\ell}}{\log{x}}\Bigr)F_{\ell ,j'}\Bigl(\frac{\log{d_\ell'}}{\log{x}}\Bigr)G\Bigl(\frac{\log{e_\ell}}{\log{y}}\Bigr)G\Bigl(\frac{\log{e_\ell'}}{\log{y}}\Bigr)}{\phi([\mathbf{d},\mathbf{d}',\mathbf{e},\mathbf{e}',\mathbf{a}])}.
\end{align}
Here we have written $\sum^*$ to indicate that the summation is restricted to the conditions that $d_k=d_k'=e_k=e_k'=1$, that $d_1d_1',\dots,d_k d_k',P_w$ are pairwise coprime, that $e_1e_1',\dots,e_k e_k',P_w$ are pairwise coprime, and that the divisibility constraints \eqref{eq:Div1} are satisfied. Similarly we have written $\sum^{**}$ for these constraints with the conditions on $d_k,d_k',e_k,e_k'$ dropped (because we have separated their contribution).

We can evaluate this by an essentially identical argument to that used in Lemma \ref{lmm:ConstSum}. The presence of the Euler $\phi$ function in the denominator affects $K_p$ by a factor $1+O_k(p^{-2})$, and so has a negligible effect. The presence of the $\mathbf{a}$ factor in the denominator means that for any $p|a_{i,j}$ we have $K_p\ll_k p^{-1}$. This means that the contribution to \eqref{eq:qSum} when $a_{1,2},\dots,a_{k,k-1}\ne 1,\dots,1$ is
\begin{equation}
\ll_k \frac{\mathfrak{S}_{m,p_0,h_k}^{(2)}|\mathcal{I}_m|}{\phi(P_w)(\log{x})^{k}(\log{y})^{k-1}}\prod_{p|a_{i,j}\text{ for some $i,j$}}\frac{O_k(1)}{p},\label{eq:ABound}
\end{equation}
where
\begin{align}
\mathfrak{S}_{m,p_0,h}^{(2)}&=\prod_{p\le w}\Bigl(1-\frac{1}{p}\Bigr)^{-(2k-2)}\prod_{w<p\le y}\Bigl(1-\frac{\omega^{(2)}_{m,p_0,h}(p)}{p}\Bigr)\Bigl(1-\frac{1}{p}\Bigr)^{-(2k-2)},\label{eq:Smph2Def}\\
\omega^{(2)}_{m,p_0,h}(p)&=\#\{1\le n\le p:p_0+(h_i-h)n\equiv 0\pmod{p}\nonumber\\
&\qquad\qquad\text{ or }m(p_0+(h_i-h)n)\equiv 1\pmod{p}\text{ for some $i$}\}.
\end{align}
Thus \eqref{eq:ABound} shows that the contribution from all $a_{1,2},\dots,a_{k,k-1}\ne 1,\dots,1$ is
\begin{align}
\sum_{\substack{a_{1,2},\dots,a_{k,k-1}|P_{z_0}/P_{w}\\ a_{1,2},\dots,a_{k,k-1}\ne 1,\dots,1}}&\frac{(-2k)^{\omega([\mathbf{a}])}}{[\mathbf{a}]}\nonumber\\
&\times\sum_{\substack{q\in\mathcal{I}_m\text{ prime}\\ q\equiv w_0\text{ (mod $P_w$)}\\ a_{i,j}|mq(h_i-h_j)-1\,\forall i\ne j}}\Bigl(\sum_{\substack{d_1,\dots,d_k\\ d_i|p_0+(h_i-h_k)q}}\sum_{\substack{e_1,\dots,e_k\\ e_i|m(p_0+(h_i-h_k)q)-1}}\lambda_{d_1,\dots,d_k,e_1,\dots,e_k}\Bigr)^2\nonumber\\
&\ll_k \frac{\mathfrak{S}_{m,p_0,h_k}^{(2)}|\mathcal{I}_m|}{\phi(P_w)(\log{x})^{k}(\log{y})^{k-1}}\Bigl(\prod_{w<p\le z_0}\Bigl(1+\frac{O_k(1)}{p^2}\Bigr)-1\Bigr)\nonumber\\
&=o_k\Bigl(\frac{\mathfrak{S}_{m,p_0,h_k}^{(2)}|\mathcal{I}_m|}{\phi(P_w)(\log{x})^k(\log{y})^{k-1}}\Bigr).\label{eq:BadPrimesSmall}
\end{align}
Hence the main contribution comes from when $a_{1,2}=\dots=a_{k,k-1}=1$. In this case, the contribution is 
\begin{align}
&\frac{(1+o_k(1))\mathfrak{S}_{m,p_0,h_k}^{(2)}|\mathcal{I}_m|G(0)^{2}}{\phi(P_w)(\log{x})^{k}(\log{y})^{k-1}}\sum_{j=1}^J\sum_{j'=1}^J F_{k,j}(0)F_{k,j'}(0)\Bigl(\int_0^\infty G'(t)^2 dt\Bigr)^{k-1}\nonumber\\
&\qquad\qquad\times\prod_{\ell=1}^{k-1}\int_{0}^\infty F_{\ell,j}'(t)F_{\ell,j'}'(t) dt \nonumber\\
&\qquad=\frac{(1+o_k(1))\mathfrak{S}_{m,p_0,h_k}^{(2)}|\mathcal{I}_m|}{\phi(P_w)(\log{x})^{k}(\log{y})^{k-1}}J_k^{(1)}(F)J_k^{(2)}(G).\label{eq:MainEstimate}
\end{align}
We obtain the same estimate \eqref{eq:MainEstimate} with $\mathfrak{S}_{m,p_0,h_k}^{(2)}$ replaced with $\mathfrak{S}_{m,p_0,h}^{(2)}$ for a different $h\in\mathcal{H}$. (Since $F$ is symmetric, $J_k^{(1)}(F)$ does not depend on the choice of $h\in\mathcal{H}$.) Substituting \eqref{eq:BadPrimesSmall} and \eqref{eq:MainEstimate} into our main term \eqref{eq:MainExpression2}, we obtain
\begin{align}
\sum_{q\in\mathcal{I}_m\text{ prime}}\mu_{m,q}(p_0)&\ge \frac{(1+o_k(1))m (\log{y})J_k^{(1)}(F)J_k^{(2)}(G)|\mathcal{I}_m|}{U I_k^{(1)}(F)I_k^{(2)}(G)\phi(P_w)}\nonumber\\
&\times\sum_{h\in\mathcal{H}}\mathfrak{S}_m\mathfrak{S}^{(2)}_{m,p_0,h}\sum_{\substack{w_0\text{ (mod $P_w$)}\\ 
(w_0,P_w)=1}}1.\label{eq:MainExpression3}
\end{align}
The inner sum is clearly $\phi(P_w)$. We note that for any $h\in\mathcal{H}$ we have $\omega_{m,p_0,h}^{(2)}(p)=0$ for $p\le w$ and $\omega_{m,p_0,h}^{(2)}(p)\le 2k-2$ for all $p$. Therefore, recalling the definitions \eqref{eq:SmDef} and \eqref{eq:Smph2Def} of $\mathfrak{S}_m$ and $\mathfrak{S}_{m,p_0,h}$, we have
\begin{align}
\mathfrak{S}_m\mathfrak{S}^{(2)}_{m,p_0,h}&=2^{-1}\prod_{p|m,\, p>2}\frac{p-2}{p-1}\prod_{p<w}\frac{\Bigl(1-\frac{1}{p}\Bigr)^{2}}{\Bigl(1-\frac{2}{p}\Bigr)}\prod_{w<p\le y}\frac{\Bigl(1-\frac{\omega_{m,p_0,h}^{(2)}(p)}{p}\Bigr)}{\Big(1-\frac{1}{p}\Bigr)^{2k-2}}\nonumber\\
&\gg \prod_{p|m,\, p>2}\frac{p-2}{p-1}\prod_{w<p\le y}\Bigl(1-\frac{2k-2}{p}\Bigr)\Big(1-\frac{1}{p}\Bigr)^{-(2k-2)}\nonumber\\
&\gg(1+o_k(1))\prod_{p|m,\, p>2}\frac{p-2}{p-1}.
\end{align}
Recalling the estimate on the size of $\mathcal{R}_m$ from Lemma \ref{lmm:SetSizes} then gives the result.
\end{proof}
Finally, we recall the key integral estimate from \cite{Maynard}.
\begin{lmm}\label{lmm:Smooth}
There exists a suitable choice of smooth functions $F,G$ such that
\[\frac{k J_k^{(1)}(F)J_k^{(2)}(G)}{I_k^{(1)}(F)I_k^{(2)}(G)}\gg \log{k}.\]
\end{lmm}
\begin{proof}
We choose $G(t)$ to be a smooth approximation to $1-t$ supported on $t\in [0,1]$ such that $J^{(2)}_k(G)/I_k^{(2)}(G)\ge 1-\epsilon$. We can choose such a $G$ since the set of non-negative continuous functions supported on $[0,1]$ is $L^2$- and $L^1$- dense in the set of continuous non-negative functions supported on $[0,1]$ (by the Stone-Weierstrasss theorem).

We recall the choice of function giving \cite[Proposition 4.3, (iii)]{Maynard}. Let $F_k:[0,\infty)^k\rightarrow \mathbb{R}$ and $g:[0,\infty)\rightarrow\mathbb{R}$ be defined by
\begin{align}
F_k(t_1,\dots,t_k)&=\begin{cases}
\prod_{i=1}^kg(kt_i),\qquad &\text{if }\sum_{i=1}^k t_i\le 1,\\
0&\text{otherwise,}
\end{cases}\\
g(t)&=\begin{cases}
1/(1+At),\qquad &t\in[0,T],\\
0,&\text{otherwise,}
\end{cases}
\end{align}
and where $A=\log{k}-2\log_2{k}$ and $T=(e^A-1)/A$. This is clearly a non-negative symmetric function defined on $\sum_{i=1}^kt_i\le 1$.

By \cite[\S 7]{Maynard}, we have that $J_k^{(1)}(F_k)/I_k^{(1)}(F_k)\gg (\log{k})/k$. We choose $F_{\ell,j}$ such that $F$ is a smooth approximation to $F_k(10 t_1,\dots,10 t_k)$ supported on $\{(t_1,\dots,t_k)\in [0,\infty)^k:\sum_{i=1}^k t_i\le 1/10\}$ with 
\[J_k^{(1)}(F)/I_k^{(1)}(F)\ge (1/10-\epsilon)J_k^{(1)}(F_k)/I_k^{(1)}(F_k),\]
which gives the result.

We can choose such an approximation $F$ since the set of symmetric, non-negative linear combinations of direct products of smooth compactly supported functions on $[0,1]^k$ is $L^2$- and $L^1$-dense in the set of non-negative symmetric $L^2$-integrable functions on $[0,1]^k$.
\end{proof}
\begin{proof}[Completion of the proof of Proposition \ref{prpstn:MainProp}]
Since 
\[|\mathcal{I}_m|=(1+o(1))\delta |\mathcal{R}_m|\log{x},\]
by Lemmas \ref{lmm:PrimeSum} and \ref{lmm:Smooth} we have that any prime  $p_0\in\mathcal{R}_m$ with $h_k x<p_0<U/m-h_k x$ has the expected number $\sum_{q}\mu_{m,q}(p_0)$ of times $p_0$ is chosen  is $\gg \delta \log{k}$. By Lemma \ref{lmm:SetSizes}, the number of primes $p_0\in\mathcal{R}_m$ which do not satisfy $h_k x<p_0<U/m-h_k x$ is $o_k(|\mathcal{R}_m|)$ for $m<U/z(\log_2{x})^2$. 

Thus we have shown that if we choose residue classes randomly according to $\mu_{m,q}$, for all but $o_k(|\mathcal{R}_m|)$ primes $p_0\in\mathcal{R}_m$, the expected number $\sum_{q}\mu_{m,q}(p_0)$ of times $p_0\in\mathcal{R}_m$ is chosen is $\gg \delta \log{k}$. By choosing $k$ sufficiently large in terms of $\delta,\epsilon$, we can ensure that this expectation is larger than $\log{\epsilon^{-1}}$. But, by the argument of Section \ref{sec:Probabilistic}, this means that the expected number of $p_0\in\mathcal{R}_m$ which are not in any of the chosen residue classes must be less than $\epsilon|\mathcal{R}_m|$. Therefore there must be at least one assignment of residue classes for which at most $\epsilon|\mathcal{R}_m|$ of the primes $p_0\in\mathcal{R}_m$ are not chosen. By the argument at the end of Section \ref{sec:ErdosRankin}, this implies Proposition \ref{prpstn:MainProp}, and hence Theorem \ref{thrm:MainTheorem}.
\end{proof}
\section{Acknowledgments}
The author would like to thank Andrew Granville and Daniel Fiorilli for many useful comments and suggestions. This work was conducted whilst the author was a CRM-ISM postdoctoral fellow at the Universit\'e de Montr\'eal.
\bibliographystyle{plain}
\bibliography{LargeGaps}
\end{document}